\documentclass[12pt,reqno]{amsart}

\usepackage{amssymb}
\usepackage{epsf,graphicx}
 
\newtheorem{theorem}{Theorem}[section]
\newtheorem{proposition}[theorem]{Proposition}

\theoremstyle{definition}

\numberwithin{equation}{section}

\newcommand{\Z}{\mathbb{Z}}
\newcommand{\R}{\mathbb{R}}

\newcommand{\I}{\mathbb{I}}
\newcommand{\C}{\mathbb{C}}
\newcommand{\SM}{\mathcal{M}}
\newcommand{\SO}{\mathcal{O}}

\newcommand{\PGL}{\mathrm{PGL}}
\newcommand{\SU}{\mathrm{SU}}
\newcommand{\PSU}{\mathrm{PSU}}
\newcommand{\GL}{\mathrm{GL}}
\newcommand{\SL}{\mathrm{SL}}
\newcommand{\ad}{\mathrm{ad}}
\newcommand{\Pic}{\mathrm{Pic}}

\begin{document}
\baselineskip=15pt

\title[Higgs bundles, branes and Langlands duality]{Higgs bundles, branes and 
Langlands duality}

\author[I. Biswas]{Indranil Biswas}

\address{School of Mathematics, Tata Institute of Fundamental
Research, 1 Homi Bhabha Road, Mumbai 400005, India}

\email{indranil@math.tifr.res.in}

\author[O. Garc\'{\i}a-Prada]{Oscar Garc\'{\i}a-Prada}

\address{Instituto de Ciencias Matem\'aticas, c/ Nicol\'as
Cabrera, no. 13--15, Campus Cantoblanco, 28049 Madrid, Spain}

\email{oscar.garcia-prada@icmat.es}

\author[J. Hurtubise]{Jacques Hurtubise}

\address{Department of Mathematics, McGill University, Burnside
Hall, 805 Sherbrooke St. W., Montreal, Que. H3A 2K6, Canada}

\email{jacques.hurtubise@mcgill.ca}



\date{}

\begin{abstract}
Given a compact Riemann surface $X$ equipped with
an anti-holomorphic involution and a complex semisimple
Lie group $G$ equipped with a real structure, we define anti-holomorphic
involutions on the moduli space of $G$-Higgs bundles over $X$. 
We describe how the various components of the fixed point locus match
up, as one passes from $G$ to its Langlands dual 
$^LG$. As an example, the case of $G\,=\, \SL(2,\C)$ and 
$^LG\,=\,\PGL(2,\C)$ is investigated in detail.
\end{abstract}

\maketitle

\tableofcontents

\section{Introduction} 

The moduli spaces of Higgs bundles on a Riemann surface $X$ are a very interesting source 
of mirror manifolds, in the sense of Strominger, Yau and Zaslow \cite{SYZ}. This 
was brought to light in successive papers of Hausel--Thaddeus \cite{HaT}, Hitchin 
\cite{Hi3}, and Donagi--Pantev \cite{DP}. The essential component of the geometric 
data for these manifolds, or rather their connected components, in this context is 
a pair of dual Lagrangian fibrations: two Lagrangian fibrations over the same 
base, with the two fibers over any general point of the base being dual tori.
It turns out that the moduli of Higgs pairs 
for the semisimple group $G$ and the moduli for its Langlands dual $^LG$ are 
mirrors in this sense (see \cite{GWZ} for more general results for $\text{SL}(n,{\mathbb C})$).

Kapustin and Witten, \cite{KW}, develop this duality extensively in their study of the geometric 
Langlands program, and as part of their study, highlight various interesting data that we can 
associate to these spaces, some of the most notable being {\it branes}. These are submanifolds, or 
more generally subvarieties which may have singularities, supporting a certain sheaf or gerbe, which come in two main types, 
{\it A-branes}, which are Lagrangian on their smooth locus, and {\it B-branes}, which are complex 
subvarieties.

In our cases, the moduli spaces under consideration are hyperK\"ahler, with 
complex structures $I,J,K$ of the following type. The $I$--structure is the natural one
for moduli space of pairs of the form $(E,\, \phi)$, where $E$ a holomorphic $G$-bundle and
$\phi$ is a form-valued section of the associated adjoint bundle, and the
$J$--structure refers to
the natural one on moduli space of pairs $(E,\,\nabla)$, where $\nabla$ is a flat
connection on a principal $G$--bundle $E$. Then $I\circ J\,=\, -J\circ I$, and $K$ is set
to be $I\circ J$. It is possible to 
have branes that are of type $A$ with respect to the $I$-structure, $B$ with 
respect to the $J$-structure, and so on. We will denote the type of our brane 
with respect to the $I,J,K$ structures by a triple, for example $(A,B,A)$.

A simple way of obtaining branes on the moduli space is to have them as the fixed point
sets of involutions. Indeed, this has been exploited with some success 
(\cite{BaS,BaS2,HS,BG,BGH,G,GR,GW}). In the Higgs case an early example, which has 
had a huge impact, was introduced by Hitchin \cite{Hi1}, who considered the simple 
holomorphic involution $(E,\,\phi)\,\longmapsto\, (E,\, -\phi)$, and showed that it gave, 
when moved into the $J$-picture, a real (i.e., anti-holomorphic) involution, whose 
fixed point sets were connections with values in a real form of the group $G$. In other words,
this would be a $(B,A,A)$-brane in the brane language. The general (physical) theory 
of branes in this context would tell us that in the mirror we would 
have a $(B,B,B)$-brane.

Another way of constructing involutions of the above hyperK\"ahler space
is to consider {\it real} Riemann surfaces,
that is Riemann surfaces $X$ equipped with an anti-holomorphic involution $\tau_X$.
Fixing an element $c$ of $\pi_1(G)$, denote by $\SM_c(G)$ the moduli space of Higgs
$G$--bundles of topological type $c$ on a real Riemann surface $(X,\, \tau_X)$.
Coupling $\tau_X$
with a real involution $\tau_G$ on our group $G$ gives natural involutions on the 
Higgs moduli space $\SM_c(G)$
$$
\tau_{\SM_c(G)}^\pm\,:\, {\SM_c(G)}\,\longrightarrow {\SM_c(G)}\, , \ \
(E,\,\phi)\,\longmapsto\,
(\overline{\tau_X^*E},\,\pm \overline{\tau_X^*\phi})\, .$$
The fixed point
locus of the involution $\tau_{\SM_c(G)}^+$ is an $(A,A,B)$-brane, while that of
$\tau_{\SM_c(G)}^-$ is an $(A,B,A)$-brane.

It is known that making a choice for the $\pm$, the actions are equivalent for all 
inner equivalent real structures on $G$ (see Section 3.3 of \cite{BGH2}), and so one 
is seeing a large number of components of our branes appearing, corresponding to 
different reductions of structure group over each real component of the curve, as 
well as other discrete data. Again, we refer to \cite{BGH2}.

An inner equivalence class of real structures on a group corresponds to an inner 
equivalence class of real structures on the Langlands dual group. This equivalence class, 
for example, is describable in terms of root data, which corresponds under 
Langlands duality, for example giving the same automorphisms of the Dynkin 
diagrams. Thus, on both sides of the mirror (meaning for $G$ and
$^{L}G$), there are $(A,A,B)$ branes for 
$(E,\,\phi)\,\longmapsto\, (\overline{\tau_X^*E},\, \overline{\tau_X^*\phi})$, and 
$(A,B,A)$ branes if the sign of the Higgs field is changed. These have been 
examined in some detail in \cite{BaS} and \cite{BG}.

It is thus legitimate to ask how the various components of $(A,A,B)$ and $(A,B,A)$ branes match 
up, as one passes from $G$ to $^LG$. This question can be asked in two different ways. Indeed, a 
first way is, as we have noted, to recall that the Higgs moduli spaces, under the $I$-complex 
structure for both $G$ and $^LG$ fiber over the same Hitchin base $B$, with (generic)
fibers dual 
Abelian varieties $A,\, A^\vee$, quotients of vector spaces $V,\, V^\vee$ by
mutually dual lattices $\Lambda, 
\Lambda^\vee$. A natural question that arises is
how the real structure interacts with these dual fibers, that is how 
the real components intersect the fibers. In some way, this is the most natural way of looking at 
the problem, as the description of the fibers involves an abelianization, and the Langlands 
duality is most manifest at the level of the Cartan subgroups. We will see that for the
degree zero (same as topologically trivial) connected components of the fibers, the real components for $G$ and $^LG$, which are translates 
of real tori, are in bijection, and that there is a fairly natural duality pairing between them once 
base sections of the Hitchin fibration are chosen. Indeed, for the lattices $\Lambda, 
\Lambda^\vee$ corresponding to $A,\, A^\vee$, there are 
sublattices $\Lambda_\pm, \Lambda^\vee_\pm$ on which the real structure acts by $\pm 1$, as well 
as projections $\pi_\pm\,:\, \Lambda \,\longrightarrow \,\Lambda_\pm/2$ and $\pi_\pm\,:\,
\Lambda^\vee\,\longrightarrow\, 
\Lambda^\vee_\pm/2$. The following theorem is proved (Theorem \ref{generalG} below):

\begin{theorem} 
The real components in the fibers $A, A^\vee$ are given by the $\Z/2\Z$ modules
$\frac{(\Lambda_-)/2} {\pi_-(\Lambda)} \times \frac{(\Lambda^\vee_-)/2} {\pi_-(\Lambda^\vee)}$. The 
pairing between $\Lambda$ and $\Lambda^\vee$ induces a perfect pairing 
$$\frac{(\Lambda_-)/2} {\pi_-(\Lambda)} \times \frac{(\Lambda^\vee_-)/2} {\pi_-(\Lambda^\vee)}
\,\longrightarrow\, \Z/2\Z$$
given by $(\lambda/2, \,\lambda^\vee/2)\,\longmapsto\, \lambda^\vee(\lambda)$ modulo $2$. 

The varieties $A ,\, A^\vee$ then have the same number of real components: 
that is, the real subvarieties of the corresponding Hitchin fibers for $G$ and $^LG$ all have the same number of real components, if they have real points. 
\end{theorem}

Of course one must calculate the effect of the real structures on these lattices; this can be a 
bit arduous, though in the end computable. We give in section four an example of this for the 
pair $G\,=\, \SL(2,\C)$ and $^LG\,=\, \PGL(2,\C)$ (and even then, just looking at one type of 
real structures on the base curve, those for which the quotient by the real structure is 
oriented). For $G\,=\, \SL(2,\C)$, this reproduces some results obtained by Baraglia and 
Schaposnik \cite{BaS}.

A second question is to ask what happens globally, on the full moduli
space of Higgs bundles; we will see, using 
results of \cite{BGH2} that in this case, even from the simple case of $G\,=\, 
\SL(2,\C)$ and $^LG\,=\, \PGL(2,\C)$, the real components are far from matching (see 
section 5 below):

\begin{proposition}
Let $k$ be the number of fixed circles on the base curve under the real structure, 
and suppose that the quotient of the base curve by the real involution is 
orientable. The fixed point set under the real involution in the $G\,=\, \SL(2,\C)$ 
Higgs moduli space has $2^{k-1}+1$ components. The fixed point set in the degree 
zero $\PGL(2,\C)$ Higgs moduli space has $3^k+1$ components.
\end{proposition}

That 
the match fails, or at least is far from obvious, should not be too surprising --- even 
on the level of groups, the Langlands duality story is much more complicated in the 
real case; see, e.g. the work of Nadler (\cite{Na}).

\section{The complex Higgs moduli space}

We summarize here some aspects that we require of the work of a succession of 
authors \cite{Hi2, Fal, D1, D2, Sc, DG} on the nature of the Hitchin fibration for 
a complex semisimple group, as well as work of Donagi and Pantev \cite{DP} which 
gives the general description of how the fibrations for $G$ and $^LG$ match.

Let $X$ be a compact connected Riemann surface of genus $g$,
with $g\, \geq\, 2$, and let $G$ be a connected 
complex semisimple Lie group. The holomorphic cotangent bundle of $X$ will be 
denoted by $K_X$. Fix an element $c\,\in\, \pi_1(G)$ giving a topological type of 
principal $G$ bundle over $X$. Let $\SM$ be the moduli space of $G$-Higgs bundles 
$(E,\,\phi)$, where $E$ is a holomorphic principal $G$-bundle over $X$, and 
$\phi$, the Higgs field, is a holomorphic section of $\ad(E)\otimes K_X$, the 
tensor product of the associated adjoint Lie algebra bundle $\ad(E)$ with the 
canonical bundle $K_X$ of $X$. The pair must satisfy a stability condition (see 
\cite{BGo,GGM}). The moduli space has a natural holomorphic symplectic structure. 
We denote by $\SM_c$ the subvariety of $\SM$ consisting of $G$-Higgs bundles with 
fixed topological invariant $c\in \pi_1(G)$ of the $G$-bundle.

The $G$-invariant polynomial on the Lie algebra of $\mathfrak g$ are precisely
that $W$ invariant polynomials on $\text{Lie}(T)\,=\, {\mathfrak t}$, where
$T$ is a maximal torus of $G$ and $W\,=\, N_G(T)/T$ is the corresponding Weyl group. 
A $G$-invariant polynomial on the Lie algebra of $\mathfrak g$ of degree $k$ gives, when 
applied to $\phi$, a holomorphic section of $K_X^{\otimes k}$. Doing this for a set 
of generators of the invariant polynomials gives a map
\begin{equation}\label{hg}
h_G\, =\, h\,:\, \SM \,\longrightarrow\, B \,=\, \bigoplus_i H^0(X,K_X^{k_i})
\end{equation}
which Hitchin \cite{Hi2} 
shows to be an integrable system with base a vector space $B$ and fibers which are 
proper. On the complement of a discriminant locus $\Delta$ in $B$, as is shown in 
\cite{DP}, the connected components of the fibers $h^{-1}(b)$ of the complex Higgs 
moduli are in bijection with those of $\SM$, which in turn are classified by an 
element $c$ of $\pi_1(G)$; if we restrict $h$ to $\SM_c$, both fiber and total space 
are connected. The fibers $h^{-1}(b),\ b\,\in\, B-\Delta$, are torsors of Abelian 
varieties $A(G,b)$, which can be identified with generalized Prym varieties for the 
Weyl group $W$ of $G$, acting on a $W$-Galois ramified cover (a ``cameral cover'') 
$C(b)$ over the base curve $X$.

Now consider the case of $c\,=\,0$. In this case the fibration $h$ has a natural section, the 
Hitchin section, described in \cite{Hi2}, once we choose a holomorphic line bundle
$L$ with $L^2\,=\, K_X$ (such a line bundle is called a theta-characteristic).
Choosing a principal $\SL(2,\C)$ inside $G$, the principal $\SL(2,\C)$-bundle 
$L\oplus L^*$ gives a principal $G$-bundle $E_0$ by extending
the structure group using the inclusion of $\SL(2,\C)$ in $G$. The Kostant section
$k\,:\,{\mathfrak t}/W \,\longrightarrow\, \mathfrak g$, where $\mathfrak t$ is the
Lie algebra of a maximal torus $T\, \subset\, G$ and $W\,=\, N_G(T)/T$
is the Weyl group, induces a map $ B \,\longrightarrow\, H^0(X,\,\ad(E_0)\otimes 
K_X)$, and so a section of $h$. This section then gives a base 
point in each fiber $h^{-1}(b)$.

Donagi and Pantev proved the following for Higgs bundles of degree zero:

\begin{theorem}[{\cite{DP}}]\label{thmdp}
Take $G$ to be simple, and fix degrees to be zero. The Hitchin bases $B_G$ and $B_{^LG}$
for $G$ and $^LG$
respectively are isomorphic, with a map $I\,:\, B_G \,\longrightarrow\, B_{^LG}$
sending the complement $B_G-\Delta_G$ of the discriminant locus to the complement
$B_{^LG}-\Delta_{^LG}$. The map $I$ lifts to an isomorphism $\widetilde I$ of the
universal cameral covers for $G$ and $^LG$ Higgs bundles over $X$.

For $b\,\in\, B_G-\Delta_G$, fixing base points, the fibers $h_{G}^{-1}(b)$ and 
$h^{-1}_{^LG}(I(b))$ (see \eqref{hg}) are dual Abelian varieties:
$$ A(^LG,I(b))\,:=h^{-1}_{^LG}(I(b)) \, \,=\, A(G, b)^\vee\,=\,
h_{G}^{-1}(b)^\vee\, .$$
\end{theorem}

We note that the cameral covers for are Weyl-invariant curves in the total space 
of the tensor product $K_X\otimes {\mathfrak t}$ over $X$, and the isomorphism of 
the cameral covers is mediated through an isomorphism $ {\mathfrak t}\,\longrightarrow\, 
^L{\mathfrak t}$ (Cartan subalgebra of $\text{Lie}({^LG})$), given by an invariant pairing.

\section{The real structures}

Now let us suppose that $X$ is a real curve, that is the Riemann surface is
endowed with an 
anti-holomorphic involution $$\tau_X\,:\, X\,\longrightarrow\, X\, .$$
We also suppose that $G$ is endowed with a real form, an anti-holomorphic group
involution $$\tau_G\,:\, G\,\longrightarrow\, G\, .$$
These are classified up to conjugation by Cartan \cite{C} as follows. One has for $G$ the 
compact real form, given by an involution $\sigma$, which can be chosen to 
preserve a Cartan subgroup $T$; the other real forms can be obtained as 
compositions $J\circ\sigma$, where $J$ is an automorphism of order two, commuting 
with $\sigma$ and also preserving $T$. The real structures on $G$
can be classifies by the outer equivalence class of $J$; we will call the set of real
structures (up 
to conjugation by inner automorphisms of $G$) corresponding to a given outer 
equivalence class a {\em clique} (see \cite{GR} for details). The set of outer 
equivalence classes is in bijection with symmetries of the Dynkin diagram.

Given $\tau_G$, and a holomorphic principal $G$-bundle $E$ on $X$, for any
principal $G$-bundle $E$, define $\overline{E}$
as the associated $C^\infty$ principal $G$-bundle
$$\overline{E} \,=\, E\times^{\tau_G}G\, .$$
and now it is possible to combine this with the involution on the curve, giving an
involution on holomorphic principal $G$-bundles:
$$E\,\longmapsto \,\tau^*_X\overline{E}\, .$$
Note that this involution of the holomorphic principal $G$-bundles is anti-holomorphic.
The real structure $\tau_G$ induces a conjugate linear involution on the Lie
algebra. Consequently, there are the following two involutions on the Higgs fields:
$$\phi\,\longmapsto\, \pm\tau_{\mathfrak g}(\tau^*_X\phi).$$
 Thus, these involutions induce involutions on the Higgs moduli spaces:
 \begin{align}
\tau_{\SM_c}^\pm\,:\, \SM_c &\,\longrightarrow\, \SM_c\nonumber\\
 (E,\,\phi)&\,\longmapsto\, (\tau^*_X\overline{E},\,\pm 
\tau_{\mathfrak g}(\tau^*_X\phi))\, .\nonumber
 \end{align}

The involutions on $ \SM_c$ induced by outer equivalent real structures (i.e., one 
is the composition of the other with an inner involution: $\tau'= J\circ\tau, 
[J,\tau]=0$) are the same. Indeed, if the involution $J$ is realized by 
conjugating by an element $g_J$, with $g_J^2 \,=\, e$, and $g_J$ fixed by $\tau$, then
there is an isomorphism between the bundles $\overline{E}^\tau$ and 
$\overline{E}^{\tau'}$ defined by $(p,\,g)\,\longmapsto\, (p,\, g_Jg)$. This isomorphism 
intertwines the actions of $\tau, \tau'$.

On the level of bundles, for $g\,>\,2$, it was shown in \cite{BGH} that the components 
of the fixed points of the involution are characterized by topological data: the 
choice of a class in $H^1(\Z/2\Z,\, G)$ for each real component of the curve, possibly 
some generalized Stiefel--Whitney class for each real component of the curve, and a 
class in $H^2(\Z/2\Z,\, Z)$, where $Z$ is the center of $G$. The class in $H^1(\Z/2\Z,\, 
G)$, in particular, determines to which real structure $I\tau_G$ in the clique the 
bundle reduces over the real component.

The outer automorphism classes of $G$ and $^LG$ are the same. What is more, the 
real structure can be made to preserve a Cartan subgroup $T$ of $G$. On the dual 
Cartan subgroup $T^\vee$ of $^LG$, there is then a dual real structure. We then 
can choose a real form for $^LG$ whose restriction to $T^\vee$ is this dual real 
structure. We then have real structures, which are duals of each other:
\begin{align}
\tau_T: T&\longrightarrow T \nonumber\\
\tau_{T^\vee}: T^\vee&\longrightarrow T^\vee.\nonumber
\end{align}
These are restrictions of real structures on $G,\, ^LG$:
\begin{align}
\tau_G\,:\, G&\,\longrightarrow\, G\nonumber\\
\tau_{^LG}\,:\, {^LG}&\,\longrightarrow\, {^LG}\,.\nonumber
\end{align}
 Combining this with the involutions on $X$, and fixing topological invariants $c,c'$ in the respective fundamental groups (these elements are necessarily invariant under the real structure), we get involutions on the respective moduli spaces:
 \begin{align}
\tau_{\SM_c(G)}^\pm\,:\, {\SM_c(G)}&\,\longrightarrow\, {\SM_c(G)}\nonumber\\
\tau_{\SM_{c'}(^LG)}^\pm\,:\, {\SM_{c'}(^LG)}&\,\longrightarrow\, {\SM_{c'}(^LG)}\, .
\nonumber
\end{align}
When $c\,=\,c'\,= \,0$, these commute with the involution on the Hitchin base $B\,=\,
B(G) \,=\,B(^LG)$, which is intertwined by the map $I$ (see Theorem \ref{thmdp})
$$
\tau_{B }\,:\, {B }\,\longrightarrow\, B\, . 
$$

Over the fixed point locus on the Hitchin base, we have involutions on the Hitchin
fibers. Indeed, the involution on $X$, combined with that on the Lie algebra $\mathfrak t$ of $T$ (there is a sign here, depending on the $\pm$ on the involution on the Higgs fields), gives an involution on $K_X\otimes 
{\mathfrak t}$ which restricts to the cameral cover $C$ of $X$. This then induces involutions on the Hitchin fibers. The connected components of the Hitchin fibers are torsors over the Abelian varieties
$A(G,b)$, $ A(^LG,I(b)) \,= \,A(G, b)^\vee$ respectively, and if there is a fixed point under $\tau_{\SM_{c }(G)}$ on the fiber for $G$, then taking this as a base point, the involution on the fiber is given by the natural involution
$$
\tau_{A(G, b)}\,:\, {A(G, b)}\,\longrightarrow\, {A(G, b)}\, .
$$
There are similar involutions of the fiber of the Hitchin map for the dual group, in
particular:
$$
\tau_{A(^LG,I(b))}\,:\, {A(^LG,I(b))}\,\longrightarrow\, {A(^LG,I(b))}\, .
$$

\begin{proposition}
Suppose that we are at a real point $b$ on the base $B$; $I(b)$ is also a real point.
The involution $\tau_{A(G, b)}$ on the Abelian variety $A(G, b)$ is determined by the involution $\tau_C$ on the cameral curve, and the involution $\tau_T$ on the Cartan subgroup $T$

In particular, the involutions $\tau_{A(G, b)}$ and $\tau_{A(^LG,I( b))}$ are dual to
each other.
\end{proposition}

\begin{proof}
The uniqueness up to scale of the invariant pairing $\langle,\,\rangle$ on
$\mathfrak t$ tells 
us that it is possible to choose it so that $\langle t,\, t'\rangle\, =\,
\overline{\langle\tau_{\mathfrak 
t}(t),\,\tau_{\mathfrak t}(t')\rangle}$, and as the pairing defines the map $I$ this then implies that if $b$ is a real 
point, $I(b)$ is also. As the cameral covers are constructed inside $K_X\otimes 
{\mathfrak t}$, the isomorphism of cameral covers between $C(b)$ and $C(I(b))$ 
intertwines the real structure. Note that $A(G,b)$ is a Prym variety for the 
action of the Weyl group on $T$-bundles on $C$, and so the real structure on 
$A(G,b)$ is induced by that on $K_X\otimes {\mathfrak t}$. Since the duality of 
the Prym varieties is also induced through the duality at the level of ${\mathfrak 
t}$, the real structure on the dual variety $A(^LG,I( b))$ is the dual of the 
structure on $A(G, b)$.
\end{proof}

Let us write a real Abelian variety $A$ of dimension $n$ as the quotient of a vector space $V$ by a lattice $\Lambda$. The real involution $\tau_A$ lifts to an anti-linear involution $\tau_V$ on $V$, which restricts to an involution $\tau_\Lambda$ on $\Lambda$. The involution $\tau_V$ splits $V$ into $+1, -1$ eigenspaces $V_+, V_-$ of real dimension $n$. 
$$V= V_+ \oplus V_-$$
Let $\Lambda_+, \Lambda_-$ be the intersections of $\Lambda$ with $ V_+,V_-$. There are projections $\pi_+ = (\I + \tau_V)/2, \pi_- = (\I - \tau_V)/2$ onto $V_+, V_-$.

\begin{proposition}The components of the fixed point set of $\tau_A$ are classified by the $\Z/2\Z$ vector space 
$$\frac{(\Lambda_-)/2} {\pi_-(\Lambda)}$$
\end{proposition}

\begin{proof}
The real components on $A$ lift to subspaces of $V$, all translates of $V_+$, and 
determined by their intersection with $V_-$. This intersection is given by $\pi_-(v) 
\,=\, (v-\tau_V(v))/2$, for $v$ in one of these components. Since $v$ under 
projection to $A$ is fixed by $\tau_A$, it follows that $ v-\tau_V(v)$ lies in 
$\Lambda_-$. On the other hand, if the component passes through the origin in $A$, 
then it follows that $v\,\in\, \pi_-(\Lambda)$.
\end{proof}

Now let us consider two real Abelian varieties $A$ and $A^\vee$, with involutions $\tau_A, 
\,\tau_{A^\vee}$ and dual lattices $\Lambda,\,\Lambda^\vee$.

\begin{theorem}\label{generalG}
The pairing between $\Lambda$ and $\Lambda^\vee$ induces a perfect pairing 
$$\frac{(\Lambda_-)/2} {\pi_-(\Lambda)} \times \frac{(\Lambda^\vee_-)/2} {\pi_-(\Lambda^\vee)}
\,\longrightarrow\, \Z/2\Z$$
given by $(\lambda/2, \,\lambda^\vee/2)\,\longmapsto\, \lambda^\vee(\lambda)$ modulo $2$. 

Thus, the varieties $A(G,\, b),\, A(^LG,\, I(b))$ have the same number of real components.
 
The real fixed locus for the Hitchin fibers in topological charge zero have the same number of real components. 
\end{theorem}

\begin{proof} 
One checks that the pairing is well defined. The involution makes $\Lambda, \Lambda^\vee$ into 
$\Z/2\Z$-modules over the integers. These modules decompose into sums of three different types of 
modules: the trivial module of rank one, the sign module, also of rank one, and the permutation 
module, of rank two (see \cite[p.~507, Theorem~74.3]{CR}). A module and its dual have the same 
decomposition. We find that each factor contributes either nothing to the quotient $\frac{(\Lambda_-)/2} 
{\pi_-(\Lambda)}$, for the trivial and permutation module, or a $\Z/2\Z$ for the sign module. The 
rank of $\frac{(\Lambda_-)/2} {\pi_-(\Lambda)}$ is then the number of sign representations in the 
decomposition. It is now
straight-forward to check explicitly for the sign module that the pairing is perfect.
\end{proof}
 
\section{An example: $\SL(2)$ and $\PGL(2)$}
 
We now consider the Higgs moduli for $\SL(2,{\mathbb C})$ and $\PGL(2,{\mathbb C})$, both to 
illustrate the preceding theorem for the Hitchin fibers. The computations of the numbers of real components for the $\SL(2,{\mathbb C})$-space are also found in the paper of Baraglia and Schaposnik \cite{BaS}, in somewhat different form, though in the end, of course, the gist is the same.

 {\it The complex moduli spaces.} 
 The structure of the $\SL(2,{\mathbb C})$ and $\PGL(2,{\mathbb C})$ moduli spaces is well known, 
and what follows is already in Hitchin \cite{Hi1}, \cite{Hi2}.

For $\SL(2,{\mathbb C})$, the moduli space is connected, and is given as a space of holomorphic 
pairs
$$
\SM_{\SL(2)} \,=\, \{ (E, \,\phi)\,\mid\, {\rm rank}(E)\,=\, 2, \,\wedge^2 E \,=\,{\mathcal 
O}_X, \,\phi\,\in\, H^0(X,\,{\rm ad}(E)\otimes K_X)\}\, ,$$
where ${\rm ad}(E)\,\subset\, \text{End}(E)$ is the subbundle of rank three defined by the
sheaf of endomorphisms of trace zero. There is a semi-stability condition \cite{Hi1}; we note 
however that we are only considering pairs with a smooth spectral curve. As destabilizing requires 
a sub-object, the spectral curve would have to have components for the pair to be unstable, and so 
as long as we only consider smooth spectral curves (the complement of the discriminant locus), all our objects will 
be stable.

For $\PGL(2,{\mathbb C})$ the moduli space is 
$$\SM_{{\rm PGL}(2)} \,=\,
\{ (E, \,\phi)\,\mid\, {\rm rank}(E)\,=\, 2, \,\phi\,\in\, H^0(X,\,{\rm ad}(E)\otimes K_X)\}/
(E\,\sim\, E\otimes L)\, ,$$
where $L$ runs over holomorphic line bundles on $X$. The moduli space $ \SM_{\PGL(2)}$ has two
components, corresponding to whether the degree of $E$ is odd or even. Again, with a smooth
spectral curve, our pairs will be stable. 

These moduli spaces are both of complex dimension $6g-6$, and fiber over the same Hitchin base. As
$\phi$ has trace zero, the conjugation invariant functions are generated by the determinant, and
$\det(\phi)$ is a quadratic differential. The Hitchin base in both cases is then 
$$B\,=\, H^0(X,\,K^{\otimes 2}_X) \,=\, \C^{3g-3}\, .$$
The cameral covers in these cases are also spectral curves, given at a quadratic differential $b$ as
a curve $S$ over $X$ in the total space of $K_X$ satisfying $\eta^2-b(x)\,=\,0$, where $\eta$ is the
tautological section the pullback of the line bundle $K_X$ to the total space of $K_X$. The curve has a natural
holomorphic involution $i$, given by 
$$i(\eta,\,x) \,=\, (-\eta,\,x)\, ;$$
thus $S$ is a two to one branched Galois cover of $X$, with $\pi\,:\,S\,\longrightarrow \,X$
denoting the projection. The discriminant locus $\Delta$ in this case is simply the variety of
quadratic differentials with multiple zeros. 

The Hitchin fiber for $\SL(2, {\mathbb C})$, for $b$ in $B-\Delta$ is a given by the variety of
line bundles $L$ on $S$ satisfying $\wedge^2\pi_*L \,=\, 
{\SO}_X$; it is a torsor for the Prym variety
$$A_{\SL(2)} \,= \,\{L\,\in\, \Pic^0(S)\,\mid\, L\otimes I^*L \,=\, {\SO}_S \}\, .$$
The Hitchin fiber for $\PGL(2, {\mathbb C})$ at $b$ is also a torsor, this time for the Abelian variety
$$A_{\PGL(2)} \,=\, \Pic^0(S)/\Pic^0(X)\, .$$
These are dual Abelian varieties.

 {\it Adding in a real involution}.
There are a variety of real structures on a curve, distinguished by the number of 
fixed (real) circles, and whether or not the quotient surface is orientable. As this 
is an illustrative example, we will consider just a type I curve, in the nomenclature 
of \cite{BHH}, that is a curve $X$ whose quotient $X_0$ by the real involution is 
orientable, and has $k$ fixed circles. The inequality $k\,\leq\, g+1$
holds; it is known as 
Harnack's inequality. The involution lifts to the canonical bundle, with, over the 
real circles, fixed points corresponding to real differential forms. It may
be noted that it is possible
to write 
the curve $X$ as the union of two oriented surfaces with boundary $X_0\bigcup 
\tau_X(X_0)$, along their common boundary of real circles.

For $\SL(2, {\mathbb C}),\, \PGL(2, {\mathbb C})$, unlike the other Lie groups in the $A$ series, there is only one clique of real structures; up to equivalence, there are two real structures, the first simply given by complex conjugation, and the second by composing conjugation with the inner automorphism given by the anti-diagonal matrix with entries $(1,-1)$. The real subgroups corresponding to these are, up to equivalence, $\SL(2,\R)$ and $\SU(2)$. We choose the real structure given by complex conjugation:
$$\tau_{\SL(2)}(g) \,=\, \overline{g}\, ,$$
with a similar formula for $\PGL(2, {\mathbb C})$.

If we now have a real Higgs bundle, that is invariant under the simultaneous 
action of the real structures on the group and its Lie algebra, and the curve, we 
have a spectral curve $S$ over $X$, cut out by $\eta^2-b(x) \,=\, 0$, with $b(x)$ 
a real quadratic differential. The zeroes of $b$ are either interchanged by the 
real structure, and so come in pairs, or lie on on real components of $X$. Again, 
as this is an example, we will just consider the case of the zeroes coming in 
pairs exchanged by the real structure. The more general case, when there are 
zeroes on the real curve, is discussed in \cite{BaS}.

The involution $\tau_X$ lifts to the spectral curve associated to a real
quadratic differential (because the real structure on $K_X$ is given by a
real structure on $X$):
$$\tau_S\,:\,S\,\longrightarrow\, S\, .$$
For the same reason, $\tau_S$ commutes with $i$. Therefore, the composition
$i\circ\tau_S$ is an anti-holomorphic involution of $S$. The real components of $X$, as they contain no zero of $q$, each lift to two 
circles on $S$. If, in a real trivialization, $q(x) \,>\,0$ on a real component, the two circles 
are fixed by $\tau_S$ and interchanged by $I$; if $q(x)\,<\,0$, then the two components are 
interchanged by both $\tau_S$ and $i$, and so fixed by $i\circ \tau_S$.

We will be concerned with the effect of the involution on the lattice $H^1(S,\, 
\Z)$, as well as the Prym sub-lattices and the dual quotient lattice; as we have 
seen, our result does not change if we consider the dual lattice $H_1(S,\, \Z)$ 
instead. We can give an explicit basis of cycles, which decomposes naturally under 
the action of $i,\, \tau_S$.
 
We start with the ``half-base curve'' $X_0$. This is an oriented surface of genus $h$, with $k$ 
real boundary circles. The genus $g$ of the surface $X$ is then $2h+k-1$. We write $X_0$ as a 
union of three pieces: the first, $X_{0,1}$ is a surface of genus $h$, with one boundary 
component; the second, $X_{0,2}$, is a cylinder; the third, $X_{0,3}$ is a punctured sphere with 
$k+1$ boundary circles, with $k$ of them being our real curves. The three components are glued 
sequentially.
 
We choose a standard homology basis $\alpha_j$, $j \,=\,1,\,\cdots ,\, 2h$, of 
$H_1(X_{0,1},\, \Z)$; let $\gamma_1,\,\cdots,\,\gamma_k$ denote the real boundary 
curves of $X_{0,3}$. The $\alpha_j$ and $\gamma_j$ generate the homology of $X_0$,
though one of the boundary cycles is redundant. The homology of $X_0$ relative to 
its boundary is generated by the $\alpha_i$, and paths $\eta_{i,j}$ joining 
$\gamma_i$ to $\gamma_j$; again there is redundancy in the latter, so that for 
example $\eta_{1,j}$, $j \,=\, 2,\,\cdots,\,k$ would suffice.

Now let us consider the branched cover $S_0$ of $X_0$. The real quadratic 
differential has $2g-2$ of its zeroes in $X_0$, and the other $2g-2$ zeroes in 
$\tau_X(X_0)$. Label the zeroes in $X_0$ as $b_1,\,\cdots,\,b_{2g-2}$; we can 
suppose that they are all in a disk $D$ living in the cylinder $X_{0,2}$. Put all 
the zeroes $b_1,\cdots,b_{2g-2}$ consecutively along a segment $L$ in $D$, and 
choose cycles $\beta_i$, $i \,=\, 1,\,\cdots,\,2g-4$, in $D$ surrounding the 
segment $(b_i,\, b_{i+1})$. One then has a branched cover $S_{0,2}$ over $X_{0,2}$, 
branching at the $b_i$, with $4$ boundary components. The full branched cover 
$S_0$ over $X_0$ is built from this by glueing two copies of $X_{0,1}$ to two of 
the boundary circles, and two copies of $X_{0,3}$ to the two others, in the 
obvious way.
 
The cycles $\beta_i$ lift to cycles $\widehat \beta_i$ on $S_0$, with 
$i_*(\widehat \beta_i)\,=\, - \widehat \beta_i$; the cycles $\alpha_i$ lift to cycles 
$\widehat \alpha_i, i_*(\widehat\alpha_i)$, the boundaries $\gamma_i$ lift to 
boundaries $\widehat\gamma_i, i_*(\widehat\gamma_i)$ and the relative classes 
$\eta_{i,j}$ to $\widehat\eta_{i,j},i_*(\widehat\eta_{i,j})$.
 
\begin{figure}\begin{center}
\includegraphics[width=12cm]{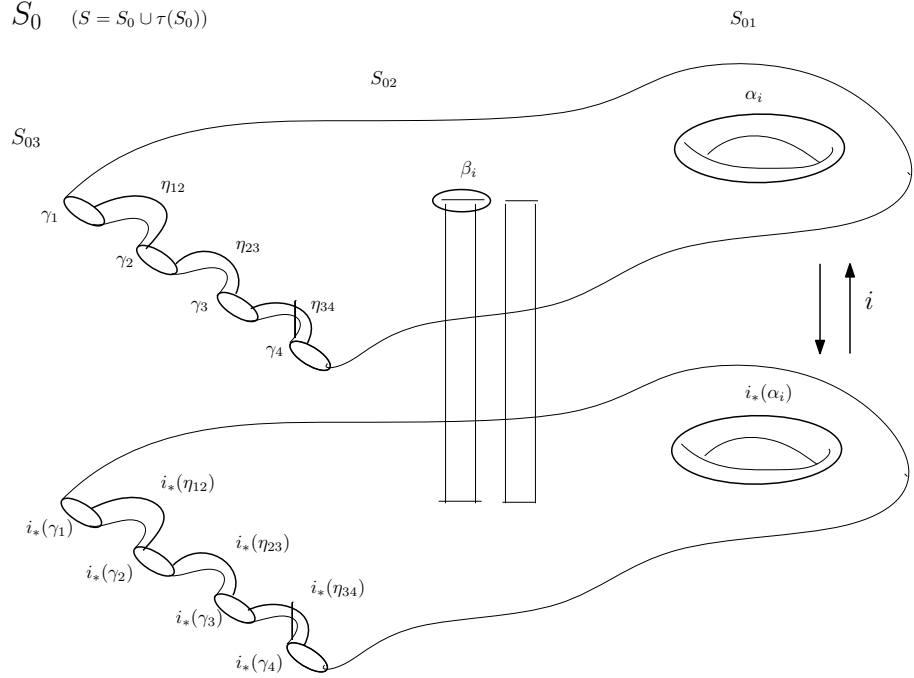}\caption{The ``half-spectral curve'' $S_0$.}\label{infinity}\end{center}\end{figure}
 
This then gives $S_0$. It must now be glued to its ``mirror image'' $\tau_S(S_0)$ 
along the real boundary circles. On each circle, one has that either 
$\widehat\gamma_i$ is glued to $(\tau_S)_*(\widehat\gamma_i)$, or to 
$i_*(\tau_S)_*(\widehat\gamma_i)$. Let us suppose that $\widehat\gamma_i$ is glued 
to $(\tau_S)_*(\widehat\gamma_i)$, for $i\,=\, 1,\,\cdots ,\,\ell$, and 
$\widehat\gamma_i$ is glued to $i_*(\tau_S)_*(\widehat\gamma_i)$ for $i\,=\, 
\ell+1,\cdots,k$. Again, as we are working through an example, we are supposing 
that $\ell$ is non-zero, and less than $k$.
 
The total curve $S$ can then be thought as five pieces lying over a line; the 
first, $S_{0,1}$ is a union of two copies of $X_{0,1}$; the second is our surface 
$S_{0,2}$, which has genus $g-2$, and four boundary components; the third is the 
union $S_{0,3}$ of two copies of $X_{0,3}$ with two copies of $\tau_S(X_{0,3})$. 
The hypothesis on $\ell$ tells us that $S_{0,3}$ is connected, with genus $2k-3$, 
and $4$ boundary components. The fourth and fifth components are the mirror images 
$S_{0,4}\,=\, \tau_S(S_{0,2}),\, S_{0,5}\,=\, \tau_S(S_{0,1})$.
 
One then has a basis of cycles on $S$:
 \begin{itemize}
\item Quadruplets of independent cycles $\widehat\alpha_j, i_*(\widehat\alpha_j), 
\tau_*(\widehat\alpha_j), \tau_*i_*(\widehat\alpha_j)$, $j= 1,\cdots,2h$.

\item Pairs of independent cycles $\widehat\beta_j, \tau_*(\widehat\beta_j)$ in 
$S$, with $i_*(\widehat\beta_j)= -\widehat\beta_j$, $j= 1,\cdots,2g-3$.

\item Pairs of independent real cycles $\widehat\gamma_j, i_*(\widehat\gamma_j)$, 
with $\tau_*(\widehat\gamma_j)=\widehat\gamma_j$, $j= 1,\cdots, \ell-1$.

\item Pairs of independent cycles $\widehat\gamma_j, i_*(\widehat\gamma_j)$, with 
$\tau_*(\widehat\gamma_j)=i_*\widehat\gamma_j$, $j= \ell+1,\cdots, k$.

\item An extra real cycle $\widehat\gamma_\ell$.

\item Pairs of independent cycles $\widehat\delta_j, i_*(\widehat\delta_j)$, $j= 
1,\cdots,\ell-1$, given by $\widehat\delta_j =\widehat \eta_{j, j+1} - \tau_* 
(\widehat\eta_{j, j+1})$, so that $(\tau_S)_*(\widehat\delta_j) = 
-\widehat\delta_j$.

\item Pairs of independent cycles $\widehat\delta_j, i_*(\widehat\delta_j)$, $j= 
\ell+1,\cdots, k-1$, given by $\widehat\delta_j =\widehat \eta_{j, j+1} - 
i_*\tau_* (\widehat\eta_{j, j+1})$, so that $(\tau_S)_*(\widehat\delta_j) = 
-i_*(\widehat\delta_j)$.

\item Finally, one takes $\eta_{\ell,\ell+1}$, and deforms $\tau( 
\eta_{\ell,\ell+1})$ to $\widetilde{\tau( \eta_{\ell,\ell+1})}$, for example 
moving it through $\tau(b_1)$, so that $\eta_{\ell,\ell+1} - \widetilde{\tau( 
\eta_{\ell,\ell+1})}$ lifts to a cycle $\widehat\delta_\ell$. The projection 
$\pi_*(\widehat\delta_\ell)$ of $\widehat\delta_\ell$ is anti-invariant under 
$\tau$, so $\pi_*(\widehat\delta_\ell+\tau_*(\widehat\delta_\ell)) =0$, giving 
$\widehat\delta_\ell +\tau_*(\widehat\delta_\ell) = \widehat\mu$, with 
$i_*(\widehat\mu) = -\widehat\mu$. Now note that there are the independent cycles 
$\widehat\delta_\ell, i_*(\widehat\delta_\ell), \widehat\mu$.
\end{itemize}
As a quick check, this gives $8g-6$ generators, while the genus of $S$ is $4g-3$. 
We have written our basis so that it is decomposed into representations of the 
group $\Z/2\Z \times \Z/2\Z$ generated by $\tau_S, i$, apart from the cycle 
$\widehat\gamma_\ell$ which in any case is fixed by $\tau$.
 
A a warm-up exercise, we can now compute the number of real components in the 
Jacobian of $X$ (this would give a count for real components in the $\GL(2, 
{\mathbb C})$ Higgs fiber). Indeed, as noted above, it is necessary to count the number of 
sign representations of the $\Z/2\Z$ generated by $\tau$ in our list above; these 
correspond to the cycles $\widehat\delta_j, j = 1,\cdots,\ell-1$, and their images 
under $i$ and the cycle $2\widehat\delta_\ell -\widehat\mu$, giving $2^{2\ell-1}$ 
components in all. This coincides with the count of \cite[Section 8.1]{BaS}.

{\it Real components of the Hitchin fiber: $\SL(2)$ and ${\rm PGL}(2)$ cases.}
We can now also look at the real connected components for the Prym variety for 
$S$, which is the Hitchin fiber for $\SL(2, {\mathbb C})$, and of the quotient 
variety ${\rm Pic}^0(S)/{\rm Pic}^0(X)$, which is the Hitchin fiber for $\PGL(2, {\mathbb 
C})$. We note first that by general results, the connectivity of the Hitchin fiber 
reflects that of the full moduli space, and so our complex Prym variety is 
connected. We are thus interested in applying our general theorem on the action of 
$\tau$ to the lattices $H^1(S,\,\Z)_{i,-}$, the sub-lattice of elements 
anti-invariant under $i$, and the quotient $H^1(S,\,\Z)/H^1(S,\,\Z)_{i,+}$; 
dually, to $H_1(S,\,\Z)/H_1(S,\,\Z)_{i,+}$ and $H_1(S,\Z)_{i,-}$.
 
We first look for the number of sign representations of $\tau$ on $H_1(S,\Z)_{i,-}$.
 Going back to our list of cycles, we find
 \begin{itemize} 
 \item The cycles $\alpha$ give permutation representations, and so contribute nothing.
 \item The cycles $\beta$ give permutation representations, and so contribute nothing.
 \item The cycles $\gamma$ give trivial representations, and so contribute nothing.
 \item The cycles $\delta_j -i_*\delta_j$ each give a sign representation.
 \end{itemize}
 
Thus, there are $2^k$ components in all. One can compare with the result of 
\cite[Theorem 38]{BaS} for the case when the spectral curve $S$ has no branch 
points on the real components.
 
Likewise, we may consider the action of $\tau$ on the quotient 
$H_1(S,\,\Z)/H_1(S,\,\Z)_+$, looking again for the sign representations in our bases 
of cycles. Again, it is the cycles $\delta_j$ which contribute, giving $2^k$ 
components in all, as outlined in our general result.

 {\it Variation in the number of components and real structures 
on vector bundles}
While the number of real components in the Hitchin fiber only depends on the base 
curve $X$ for the $\SL(2, {\mathbb C})$ and $\PGL(2, {\mathbb C})$ moduli, the 
number of real components for $\GL(2, {\mathbb C})$, i.e., the number of real 
components of the Picard variety on the spectral curve, does depend on the 
spectral curve and so does depend on where on the real locus of $B-\Delta$ one is; 
one has $2^{2\ell-1}$ components, where $2\ell$ is the number of real components 
of the spectral curve. (A simple example, by the way, of such a variation, is 
given by a family of elliptic curves $y^2 \,=\, x(x^2-a)$, which has two real 
components for $a\,>\,0$, and only one for $a\,<\,0$, with $a \,=\, 0$ 
corresponding to a singular curve.)

\section{Global invariants: $\SL(2,\C)$ and $\PGL(2,\C)$}

A question is to see how the enumeration corresponds to the global moduli of real 
Higgs pairs over the curve $X$. The real part of the moduli space of $G$-Higgs 
bundles contains the total space of the cotangent bundle of the real part of the 
moduli space of principal $G$-bundles as a dense open subset such that the 
complement is a closed real analytic subspace of (real) codimension at least two. 
Indeed, this follows from immediately the fact that the total space of the cotangent 
bundle of a moduli space of principal $G$-bundles is a Zariski open subset of the 
corresponding moduli space of Higgs $G$-bundles such that complement has (complex) 
codimension at least two \cite[Theorem~II.6(iii)]{Fal}. Therefore, it suffices to 
consider the moduli of bundles.
 
For real $G$ bundles, following a classification of \cite{BGH} in terms of Galois
cohomology, we have as data, following the choice of a real structure $\tau_G$ on
$G$ (note that a real structure should be considered for each clique):
\begin{itemize} 
\item The choice of a class $c$ in $H^2(\Z/2\Z,\, Z_G)$, which decides whether the 
bundle has a real or pseudo-real structure;

\item The choice over each real component $C$, of a class $h_C$ in 
$H_c^1(\Z/2\Z,\, G)$;

\item The choice over each real component, of a generalized Stiefel-Whitney class 
$w_1(C)$ in $\pi_0(Stab(h_C))$, where the action of $G$ on $h_C$ is by 
$h\,\longmapsto\, \tau_G(g)h g^{-1}$;

\item The choice of a global degree for the bundle, with the choice being 
constrained by the choice of Stiefel-Whitney classes.
\end{itemize}

The main result of \cite{BGH} is that once the topological data is fixed, there is 
a connected moduli space with the fixed data.

For $\SL(2,\C)$, with the real structure $\tau_{\SL(2)}(g) \,=\, (g^*)^{-1}$, the
following hold: 
that
\begin{itemize} 
\item There are the classes $c\,=\, \I, -\I$ in $H^2(\Z/2\Z,\, Z_G) \,=\, (Z_G)_\R/\tau_G(Z)Z$.

\item For $c\,=\,\I$, there are two classes in $h \,=\,\pm \I$ in $H^1(\Z/2\Z,\,G)$; for
$c\,=\, -\I$, there is a class $h \,=\, \begin{pmatrix} 
\sqrt{-1}&0\\0&-\sqrt{-1}\end{pmatrix}\,\buildrel{\mathrm def}\over{=}\, \sigma$ in $H_{-1}^1(\Z/2\Z,\, 
G)$.

\item For all of these cases, the stabilizers ( $\SU(2)$, $\SL(2,\,\R)$) are 
connected, so that there are no Stiefel-Whitney classes.

\item The bundles, being $\SL(2,\,\C)$ bundles, have degree zero.
\end{itemize}

Note that on each real curve, for $c\,=\,\I$, we are getting an $\SU(2)$ bundle with a 
hermitian form that is either positive or negative; one supposes one can normalize 
the sign globally, giving $2^{k-1}$ component to the bundle moduli space; for 
$c\,=\,-\I$, we are getting one component of $\SU(1,1)$ bundles; so in all there are 
$2^{k-1} + 1$ components.

For $\PGL(2,\C)$, with the real structure $\tau_{\PGL(2)}(g) = (g^*)^{-1}$:
\begin{itemize} 
\item There is a single class $c\,=\, \I$ in $H^2(\Z/2\Z,\, Z_G)$ (no center).

\item There are two classes $h \,=\,1$, $h \,=\, \sigma $ in $H^1(\Z/2\Z,\, G)$.

\item For $h\,=\,\sigma$, the stabilizer is $\PGL(2,\R)$, which has two components, and
so there is a (classical) Stiefel-Whitney class associated to each component of the
curve; for $h\,=\, \I$, the stabilizer is ${\rm SO}(3)$, and connected.

\item There is a global degree, with values in $\Z/2\Z$, which is determined by
the parity of the number of real components with non-zero Stiefel-Whitney
class, which as we are concerned with degree zero, must be even.
\end{itemize}

Counting these gives $(3^k+1)/2$ components for the moduli: for each real circle, either
$h\,= \,1$ or $h\,=\, \sigma$, and if $h\,=\, \sigma$, either a trivial or non-trivial 
Stiefel-Whitney class; in short, a reduction to either ${\rm SO}(3)$ or to $\PGL(2,\R)$, 
and in the latter case, oriented or not. This would give $3^k$ choices, but there 
is a parity constraint given by the global degree: there must be an even number of 
components with non-zero Stiefel-Whitney class. If $n_k$ is the number of 
possibilities, then we can assign the first $k-1$ choices arbitrarily, and if the 
parity constraint is satisfied for these, there are exactly two possibilities for the 
remaining choice; if it is not, there is only one. This gives an inductive formula 
$$n_k \,=\, 2n_{k-1} + 3^{k-1}-n_{k-1} \,=\, 3^{k-1} + n_{k-1}\, .$$ Solving the 
recursion gives $n_k\,=\, (3^k+1)/2$.

In the complex world this might seem paradoxical, as it gives more global 
components than components in the fiber. However, in our real situation, for a 
fixed fiber, i.e., a fixed spectral curve, at most $2^k$ of our components can 
intersect the fiber. Indeed, there are $\ell$ real components on the curve whose 
lifts $\gamma_i$ satisfy $\tau_*(\gamma_i) \,=\, \gamma_i$, and
there are $k-\ell$ components which satisfy 
$\tau_*(\gamma_i) \,=\, i_*\gamma_i$. For the latter case, it is not possible to have
${\rm PSU}(2)$ Higgs fields, as for these the eigenvalues are fixed by the real 
structure, and so is left with only two choices, $\PGL(2,\R)$ orientable or not, 
for these components. In the same vein, if $\tau_*(\gamma_i) \,= \,\gamma_i$,
then we can 
write a $\PGL(2,\R)$-bundle as a sum of real line bundles, which are either trivial 
or non-orientable. However, the determinant of the Higgs field has no zero along 
the curve, and so the determinant bundle is trivial. Thus the tensor product of 
the two eigenline bundles is trivial, and so they themselves are either both 
trivial or both non-orientable, showing the $\PGL(2,\R)$ bundle is orientable. Thus 
for these real components also there are only two choices, either $\PSU(2)$ or 
$\PGL(2,\R)$ orientable. In all, then, there are at most $2^k$ possibilities along 
each Hitchin fiber.

\section*{Acknowledgements}

We thank International Centre for Theoretical Sciences for hospitality while a part
of the work was carried out.

\end{document}